\documentclass[12pt,a4paper,twoside]{article}
\usepackage{amsmath,amsfonts,amsthm,amsopn,color,amssymb,enumitem}
\usepackage{palatino}
\usepackage{graphicx}

\newcommand{\eps}{\varepsilon}

\newcommand{\R}{\mathbb{R}}

\newcommand{\RN}{{\mathbb{R}^N}}

\renewcommand{\le}{\leqslant}
\renewcommand{\ge}{\geqslant}

\newcommand{\n }{\nabla }

\renewcommand{\H}{H^1(\RN)}
\newcommand{\Hr}{H^1_r(\RN)}

\newcommand{\D}{{\mathcal D}}

\newcommand{\irn }{\int_{\RN}}

\def\bbm[#1]{\mbox{\boldmath $#1$}}
\newcommand{\beq }{\begin{equation}}
\newcommand{\eeq }{\end{equation}}

\newtheorem{theorem}{Theorem}[section]

\newtheorem{proposition}[theorem]{Proposition}
\newtheorem{remark}[theorem]{Remark}
\newtheorem{corollary}[theorem]{Corollary}
\renewenvironment{proof}{\noindent{\textbf{Proof\quad}}}{$\hfill\square$\vspace{0.2 cm}\\}

\title{{\bf A note on the elliptic Kirchhoff \\equation in $\R^N$
perturbed by a local nonlinearity
\footnote{The author is supported by M.I.U.R. - P.R.I.N. ``Metodi
variazionali e topologici nello studio di fenomeni non lineari''}}}

\author{A. Azzollini \thanks{Dipartimento di Matematica, Informatica ed Economia, Universit\`a degli
Studi della Basilicata,  Via dell'Ateneo Lucano 10, I-85100 Potenza,
Italy, e-mail: {\tt antonio.azzollini@unibas.it}}}

\date{}

\begin{document}

\maketitle

\begin{abstract}
    In this note we complete the study made in \cite{A} on a Kirchhoff type equation with a Berestycki-Lions
    nonlinearity. We also correct
    Theorem 0.6 inside.
\end{abstract}

\section*{Introduction}

In this note we consider the nonlinear Kirchhoff equation
    \begin{equation}\label{eq:k}
            -\left(a+b\int_{\R^N}|\n u|^2\right)\Delta u = g(u)\hbox{ in }\RN,\; N\ge
            3,
    \end{equation}
where we assume general hypotheses on $g$. We will investigate the
existence of a solution depending on the value of the positive
parameter $a$ and $b$. We will fix an uncorrect sentence contained
in \cite{A} and complete that paper with additional results.\\
We refer to \cite{A} and the references within for a justification
of our study and a bibliography on the problem.

\section{Existence and characterization of the solutions}
Assume that
        \begin{itemize}
\item[({\bf g1})] $g\in C(\R,\R)$, $g(0)=0$;
\item[({\bf g2})]
$-\infty <\liminf_{s\to 0^+} g(s)/s\le \limsup_{s\to 0^+}
g(s)/s=-m<0$;
\item[({\bf g3})] $-\infty \le\limsup_{s\to +\infty}
g(s)/s^{2^*-1}\le 0$;
\item[({\bf g4})] there exists $\zeta>0$
such that $G(\zeta):=\int_0^\zeta g(s)\,d s>0$.
\end{itemize}

It is well known that the previous assumptions coincide with that in
\cite{BL1}, where the problem
    \begin{equation}\label{eq:bl}
            -\Delta v = g(v)\hbox{ in }\RN,\; N\ge
            3,
    \end{equation}
was studied and solved.\\
First of all, we present the following general result which provides
a characterization of the solutions of \eqref{eq:k}
\begin{theorem}\label{th:iff}
    $u\in C^2(\RN)\cap\D^{1,2}(\RN)$ is a solution to \eqref{eq:k} if and only if there exists $v\in C^2(\RN)\cap\D^{1,2}(\RN)$ solution to \eqref{eq:bl} and
$t>0$ such that $t^2a+t^{4-N}b\int_{\R^N}|\n v|^2=1$ and
$u(\cdot)=v(t\cdot)$.
\end{theorem}
\begin{proof}
    We first prove the  ``if'' part. Suppose $v\in C^2(\RN)\cap\D^{1,2}(\RN)$ and $t>0$ are as in the statement of the
    theorem and set $u(\cdot)=v(t\cdot)=v_t(\cdot)\in C^2(\RN)\cap\D^{1,2}(\RN)$. We compute
        \begin{align*}
            -\Delta u (x)&=-\Delta v_t(x)=-t^2\Delta v(tx)\\
            &=t^2 g(v(tx))=t^2g(u(x))=\frac{g(u(x))}{a+bt^{2-N}\int_{\R^N}|\n v|^2}\\
            &=\frac{g(u(x))}{a+b\int_{\R^N}|\n u|^2}.
        \end{align*}
    Now we prove the ``only if'' part. Suppose $u\in C^2(\RN)\cap\D^{1,2}(\RN)$ is a solution of \eqref{eq:k} and
    set $h=\sqrt{a+b\int_{\R^N}|\n u|^2}$, $v(\cdot)=u(h\cdot)=u_h(\cdot)$.
    Of course $v\in C^2(\RN)\cap\D^{1,2}(\RN)$ and $u(\cdot)=v(\frac 1 h \cdot)$. Moreover
        \begin{equation*}
            -\Delta v (x)= -h^2\Delta u (hx)= h^2\frac{g(u(hx))}{a+b\int_{R^N}|\n u|^2}=g(u_h(x))=g(v(x))
        \end{equation*}
    and, if we set $t=\frac 1 h,$
        \begin{equation*}
            t^2 =\frac 1{a+b\int_{\R^N}|\n u|^2}=\frac 1 {a+bt^{2-N}\int_{\R^N}|\n v|^2}.
        \end{equation*}
\end{proof}
    \begin{remark}\label{re:gr}
        Assume ({\bf g1}$\ldots${\bf g4}). By the existence result contained in \cite{BL1}, it is obvious that for $N=3$
        there exists a solution of \eqref{eq:k} for any $a$ and $b$ positive numbers.\\
        For $N=4,$ we should have a solution if and only if there exists $v$
        solution of \eqref{eq:bl} such that $b\irn|\n v|^2<1.$
        Taking into account the computations in \cite[Section 4.3]{BL1}, we know
        that the ground state solution of equation \eqref{eq:bl} has the minimal value of the $\D^{1,2}(\RN)$ norm
        among all the solutions of the equation. Then, for $N=4$ we conclude that  equation \eqref{eq:k}
        has a solution if and only if the ground
        state solution $\bar v$ of \eqref{eq:bl} is such that $b\irn|\n \bar v|^2<1.$
    \end{remark}
Consider the functional of the action related with equation \eqref{eq:k}
    \begin{equation*}
      I(u)= \frac 1 2 \left( a +\frac b 2\irn|\n u|^2\right)\irn|\n u|^2-\irn G(u),
    \end{equation*}
where $G(s)=\int_0^s g(z)\,dz$, being $g$ possibly modified as in
\cite{BL1} in order to make $I$ a $C^1$ functional on $\H.$ We
observe that, for small dimensions, the value of the action computed
in the solutions increases as the $\D^{1,2}(\RN)$ norm increases
according to the following
    \begin{proposition}\label{pr:mon}
        Assume $N=3$ or $N=4$. If $v_1$ and $v_2$ are solutions of \eqref{eq:bl} and $\irn |\n v_1|^2<\irn |\n v_2|^2$ and, for $N=4,$ we also have $b\irn |\n v_2|^2<1,$ then, calling $t_1$ and $t_2$ the positive numbers such that respectively $v_1(t_1\cdot)$ and $v_2(t_2\cdot)$ are solutions of \eqref{eq:k}, we have $t_2<t_1$ and $I(v_1(t_1\cdot))<I(v_2(t_2\cdot))$.
    \end{proposition}
    \begin{proof}
         By Theorem \ref{th:iff}, it is immediate to see that $t_2<t_1.$  Now observe that any solution of \eqref{eq:k} satisfies the Pohozaev identity
            \begin{equation}
                a\frac{N-2}{2N}\irn |\n u|^2
        + b\frac{N-2}{2N}\left(\irn |\n u|^2\right)^2-\irn G(u)=0.
            \end{equation}
        As a consequence the action computed in any solution of \eqref{eq:k} is
            \begin{equation*}
                I(u)=a\frac 1N\irn|\n u|^2+b\frac{4-N}{4N}\left(\irn|\n u|^2\right)^2
            \end{equation*}
        and then, if $v$ and $t>0$ are related with $u$ as in Theorem \ref{th:iff}, we have that
            \begin{equation*}
                I(u)=a\frac{t^{2-N}}N\irn|\n v|^2+b\frac{4-N}{4N}t^{4-2N}\left(\irn|\n v|^2\right)^2.
            \end{equation*}
        Since $t^2a+t^{4-N}b\int_{\R^N}|\n v|^2=1$, we can cancel the dependence of $I$ from $b$
            \begin{align}\label{eq:func}
                I(u)&=\frac{\irn|\n v|^2}{Nt^N}\left(at^2+\frac{4-N}{4}(1-at^2)\right)\nonumber\\
                &=\frac a4\frac{\irn|\n v|^2}{t^{N-2}}+\frac{4-N}{4N}\frac{\irn|\n v|^2}{t^{N}}.
            \end{align}
        The conclusion easy follows from \eqref{eq:func}.
    \end{proof}
   In the following Corollary we establish the conditions which
   guarantee the existence of a ground state solution for $N\ge 3$.
   In particular we correct Theorem 0.6 in \cite{A} for what concerns the dimension $N=4.$
    \begin{corollary}
        Assume ({\bf g1}$\ldots${\bf g4}).\\
        If $N=3$ then equation \eqref{eq:k} has a ground state solution.\\
        If $N=4$ then equation \eqref{eq:k} has a ground state solution if and only if $b\irn |\n \bar v|^2<1$, being $\bar v$ a ground state solution of \eqref{eq:bl}.\\
        If $N\ge 5$ then equation \eqref{eq:k} has a solution if and only if
            \begin{equation}\label{eq:min}
                a\le\left(\frac{N-4}{N-2}\right)^{\frac{N-2}{N-4}}
                \left(\frac 2{(N-4)b\irn|\n \bar v|^2}\right)^{\frac 2{N-4}},
            \end{equation}
        being $\bar v$ a ground state solution of \eqref{eq:bl}. Moreover the functional attains the infimum.
    \end{corollary}
    \begin{proof}
        Since the functional of the
        action related with equation \eqref{eq:bl}, when computed in the solutions
        of the equation, is directly proportional to the $\D^{1,2}(\RN)$ norm of the
        solutions (see Remark \ref{re:gr}), the conclusion for cases $N=3$ and $N=4$ follows immediately by
        Proposition \ref{pr:mon} and \cite{BL1}.\\
         If $N\ge 5,$ by Theorem
        \ref{th:iff} we have to show that there exists a solution $v$ of equation
        \eqref{eq:bl} and $t>0$ such that $t^2a+t^{4-N}b\int_{\R^N}|\n
        v|^2=1$. Of course such a couple
        $(v,t)$ exists if only if there exists $t_0>0$ such that
            \begin{equation}\label{eq:ex}
                t_0^2a+t_0^{4-N}b\int_{\R^N}|\n\bar
                v|^2=1
            \end{equation}
        for $\bar v$ ground state solution of \eqref{eq:bl}.
        By studying the function $f(t)= at^2+b_0t^{4-N}$ for $t>0$,
        being $b_0=b\int_{\R^N}|\n
        \bar v|^2,$ we observe that \eqref{eq:ex} holds for some $t_0$ if
        and only if
            \begin{equation}\label{eq:min2}
                \min_{t>0}f(t)\le 1.
            \end{equation}
        An easy computation leads
        to \eqref{eq:min}.
        As a remark we point out that if
        \eqref{eq:min2} holds with the strict inequality,
        then we can find two values $t_1<t_2$ which solve \eqref{eq:ex} and two corresponding
        distinct solutions $\bar v_{t_1}$ and
        $\bar v_{t_2}$ to equation \eqref{eq:k}.

        Now we prove that the functional $I$ attains the minimum.
        For $i=1,2$, define $g_i$ and $G_i$  as in \cite{BL1}. Observe
        that, by (3.4) and (3.5) in \cite{BL1},
            \begin{align*}
                I(u)&=\frac 1 2 \left( a +\frac b 2\irn|\n u|^2\right)\irn|\n u|^2+\irn
                G_2(u)-\irn G_1(u)\\
                &\ge \frac 1 2 \left( a +\frac b 2\irn|\n u|^2\right)\irn|\n u|^2+(1-\eps)\irn
                G_2(u)-C_\eps \irn |u|^{\frac{2N}{N-2}}\\
                &\ge \frac 1 2 \left( a +\frac b 2\irn|\n u|^2\right)\irn|\n u|^2+\frac{1-\eps}2m\irn
                |u|^2-C_\eps \irn |u|^{\frac{2N}{N-2}},
            \end{align*}
        where $\eps <1$ and $C_\eps>0$ are suitable constants.\\
        Since
        $\D^{1,2}(\RN)\hookrightarrow L^{\frac{2N}{N-2}}(\RN)$, for
        a suitable positive constant $C$ we have
            \begin{multline*}
                I(u)\ge \frac 1 2 \left( a +\frac b 2\irn|\n u|^2\right)\irn|\n u|^2\\+\frac{1-\eps}2m\irn
                |u|^2-C \left(\irn |\n u|^2\right)^{\frac{N}{N-2}},
            \end{multline*}
         and, since for
        $N\ge 5$ we have $1<\frac{N}{N-2}< 2$, we deduce that
        the functional is bounded below and coercive with respect to
        $H^1(\RN)$ norm.\\
        Now, since for every $u\in\H$ and its corresponding Schwarz
        symmetrization $u^*$ we have
            \begin{equation*}
                \irn |\n u^*|^2\le \irn |\n u|^2,\qquad \irn
                G(u^*)=\irn G(u),
            \end{equation*}
        we can look for a minimizer of $I$ in $H_r^1(\RN)$, the set of
        radial functions in $\H.$ As in
        \cite{BL1}, we can prove that the functional
            \begin{equation*}
                u\in\Hr\mapsto\irn G_1(u)\in\R
            \end{equation*}
        is compact, so, by standard arguments based on Weierstrass Theorem, $I$ attains the infimum.
    \end{proof}
    \begin{remark}
        Suppose $N\ge 5$. By previous Corollary we have that if
        \eqref{eq:min} does not hold, then $I$ is nonnegative in
        $\H.$
    \end{remark}


\begin{thebibliography}{99}
        \bibitem{A}
            Azzollini, A.: The elliptic Kirchhoff equation in $\R^N$
            perturbed by a local nonlinearity. {\it Differential and Integral
            equations}, 25: 543--554 (2012).

        \bibitem{BL1}
Berestycki, H., Lions, P.L.: Nonlinear scalar field equations. I.
Existence of a ground state. {\it Arch. Rational Mech. Anal.}, 82:
313--345 (1983).


    \end{thebibliography}
\end{document}